\documentclass[12pt,a4paper]{amsart}
\usepackage{amsmath,amssymb,amsthm,mathtools,tikz}
\usetikzlibrary{patterns}
\usetikzlibrary{arrows,decorations.markings}
\usepackage[headings]{fullpage}
\usepackage{algorithm}
\usepackage[noend]{algpseudocode}

\tikzstyle arrowstyle=[scale=1.5]
\tikzstyle directed=[postaction={decorate,decoration={markings,
    mark=at position .725 with {\arrow[arrowstyle]{stealth}}}}]

\usepackage[backend=bibtex]{biblatex}
\newtheorem{theorem}{Theorem}[section]
\newtheorem{lemma}[theorem]{Lemma}             

\theoremstyle{definition}
\newtheorem{definition}[theorem]{Definition}   
\newtheorem{remark}[theorem]{Remark}           

\theoremstyle{remark}
\newtheorem{example}[theorem]{Example}         

\newtheoremstyle{TheoremNum}
       {\topsep}{\topsep}              
        {\itshape}                      
        {}                              
        {\bfseries}                     
        {.}                             
        { }                             
        {\thmname{#1}\thmnote{ \bfseries #3}}
\theoremstyle{TheoremNum}
\newtheorem{thmn}{Theorem}

\numberwithin{equation}{theorem}

\newcommand\red{\color{red}}
\newcommand\cx[1]{\multicolumn{1}{|c|}{\!#1\!}}
\newcommand\cxr[1]{\multicolumn{1}{|c|}{\red\!#1\!}}
\newcommand\mcx[1]{\multicolumn{1}{c|}{\!#1\!}}

\newcommand\fA{\mathfrak{A}}

\newcommand\ba{\mathbf{a}}
\newcommand\bb{\mathbf{b}}
\newcommand\bc{\mathbf{c}}
\newcommand\bd{\mathbf{d}}
\newcommand\bl{\boldsymbol{\rotatebox{10}{$\ell$}}}
\newcommand\bm{\mathbf{m}}
\newcommand\bL{\mathbf{L}}
\newcommand\bv{\mathbf{v}}
\newcommand\bx{\mathbf{x}}
\newcommand\by{\mathbf{y}}
\newcommand\bw{\mathbf{w}}
\newcommand\br{\mathbf{r}}

\newcommand\Reals{\mathbb{R}}
\newcommand\N{\mathbb{N}}

\newcommand\Z{\mathbb{Z}}

\newcommand{\bnv}{\texttt{burnt\_vertices}}
\newcommand{\damp}{\texttt{dampened\_edges}}
\newcommand{\tree}{\texttt{tree\_edges}}
\newcommand{\pf}{\mathcal{P}}
\newcommand{\cR}{\mathcal{R}}

\newcommand\sbr[1]{\left\langle#1\right\rangle}
\newcommand\nbr[1]{\big|\langle#1\rangle\big|}
\newcommand\bbr[1]{\big|\big\langle#1\big\rangle\big|}

\DeclareMathOperator{\leg}{leg}
\newcommand{\T}{\mathsf{T}}

\newcommand{\PF}{\mathsf{PF}}

\DeclareMathOperator{\run}{run}
\DeclareMathOperator{\Run}{Run}

\DeclareMathOperator{\LT}{\mathcal{LL\/T}}
\DeclareMathOperator{\ZP}{\mathcal{ZP\!F}}
\DeclareMathOperator{\RP}{\mathcal{RP\!F}}
\DeclareMathOperator{\RR}{\mathcal{RR\/W}}

\newcommand{\RW}{\mathsf{RW}}

\DeclareMathOperator{\dfs}{DFS}

\DeclareMathOperator{\rim}{rim}
\DeclareMathOperator{\coim}{coim}

\newcommand\eqcl[1]{\bar{#1}}

\title{The number of parking functions with center of a given length}

\author[R.~Duarte]{Rui Duarte} \address{CIDMA and Department of
  Mathematics, University of Aveiro, 3810-193 Aveiro, Portugal}
\email{rduarte@ua.pt}

\author[A.~Guedes~de~Oliveira]{Ant\'onio Guedes de Oliveira}
\address{CMUP and Department of Mathematics, Faculty of Sciences,
  University of Porto, 4169-007 Porto, Portugal}
\email{agoliv@fc.up.pt}

\thanks{This work was partially supported by CMUP (UID/MAT/00144/2013)
  and CIDMA (UID/MAT/04106/2013), which are funded by FCT (Portugal)
  with national (ME) and European structural funds through the
  programs FEDER, under the partnership agreement PT2020.}

\begin{document}
\begin{abstract}
Let $1\leq r\leq n$ and suppose that, when the \emph{Depth-first
  Search Algorithm} is applied to a given rooted labeled tree on $n+1$
vertices, exactly $r$ vertices are visited before backtracking. Let $R$
be the set of trees with this property.  We count the number of
elements of $R$.

For this purpose, we first consider a bijection, due to Perkinson,
Yang and Yu, that maps $R$ onto the set of parking function with
\emph{center} (defined by the authors in a previous article) of size
$r$.
A second bijection maps this set onto the set of parking functions
with \emph{run} $r$, a property that we introduce here.
We then prove that the number of length $n$ parking functions with a
given run is the number of length $n$ \emph{rook words} (defined by
Leven, Rhoades and Wilson) with the same run. This is done by counting
related lattice paths in a ladder-shaped region.
We finally count the number of length $n$ \emph{rook words} with run
$r$, which is the answer to our initial question.
\end{abstract}


\maketitle

\section{Introduction}

Let $\T_n$ be the set of rooted labeled trees on the set of vertices
$\{0,1,\dotsc,n\}$ with root $r=0$, and let $T\in\T_n$.  Suppose that
the \emph{Depth-first Search Algorithm} ($\dfs$) is applied to $T$ by
starting at $r$ and by visiting at each vertex the unvisited
neighbor of highest label.  If $T$ is not a path with endpoint $r$, at
a certain moment the algorithm will backtrack. In this paper we are
concerned with the number of vertices that are visited before this
happens.

More precisely, let $\bv=\bv(T)=(v_1,\dotsc,v_k)$ be the ordered set
of vertices different from the root that are visited \emph{before
  backtracking}, and let $\leg(T)=k$ be the length of $\bv$.  We
evaluate explicitly the enumerator
$$\LT_n(t)=\sum_{T\in\T_n} t^{\leg(T)}\,.$$

It is well-known that $|\T_n|=(n+1)^{n-1}=|\PF_n|$, where $\PF_n$ is
the set of \emph{parking functions of length $n$}, consisting of the
$n$-tuples $\ba=(a_1,\dotsc,a_n)$ of positive integers such that the
$i$th entry \emph{in ascending order} is always at most $i\in[n]:=\{1,\dotsc,n\}$.
As usual, we may denote $\ba$ either as a word, $\ba=a_1\dotsb a_n$,
or as a function, $\ba\colon[n]\to[n]$ such that $\ba(i)=a_i$.

Parking functions are important combinatorial structures with
several connections to other areas of mathematics (see
e.g. Haglund \cite{Hag} and the recent survey by Yan \cite{Yan}).  In
particular, starting with Kreweras \cite{Kreweras}, various bijections between trees
on $n+1$ vertices and parking functions of length $n$ were defined where the
reversed sum enumerator for parking functions is the counterpart to the inversion
enumerator for trees \cite{Knuth,GOLV,Shin,Yan}. In a recent paper, Perkinson, Yang and Yu
\cite{PYY} constructed a very general algorithm that gives us as a
particular case a new bijection with this property.

We show in Section~\ref{algorithms} (cf. {\cite{DGO2}}) that, for
this bijection, the counterpart of the statistics $\leg(T)$ is
$z(\ba)=|Z(\ba)|$, where $Z(\ba)$ is \emph{the center of $\ba$}
defined in \cite{DGO}.
We recall that, for any $a \in [n]^n$, the center of $\ba$ is the largest subset $X = \{ x_1, \ldots, x_\ell \}$ of $[n]$ such that $1 \leq x_1 < \cdots < x_\ell \leq n$ and
$a_{x_i} \leq i$ for every $i \in [\ell]$.
We  namely prove that if $\bv(T)=(v_1,\dotsc,v_k)$ then exactly $Z(\ba)=\{v_1,\dotsc,v_k\}$.
Hence, we obtain the following result, if we consider the enumerator
$$\ZP_n(t)=\sum_{\ba\in\PF_n} t^{z(\ba)}\,,$$
\begin{thmn}[\textbf{\ref{t2pf}}]
For every $n\in\N$,
$$\LT_n(t)=\ZP_n(t)\,.$$
\end{thmn}

The evaluation of this new enumerator was indeed part of our initial
twofold purpose for its role in the theory of parking functions,
described as follows.  Consider the Shi arrangement, formed by all the
hyperplanes defined in $\Reals^n$ by equations of form $x_i-x_j=0$ and
of form $x_i-x_j=1$, where $1\leq i<j\leq n$. Let $R_0$ be the chamber
of the arrangement consisting of the intersection of all the open
slabs defined by the condition $0<x_i-x_j<1$.  Pak and Stanley
\cite{Stan} defined a bijective labeling of the chambers of this
arrangement by parking functions, in which $R_0$ is labeled with the
parking function $(1,\dotsc,1)$ and, along a shortest path from $R_0$
to any other chamber, for any crossed hyperplane of form $x_i-x_j=0$
the $j$th coordinate of the label is increased by one, and for any
crossed plane of form $x_i-x_j=1$ it is the $i$th coordinate that is
increased by one. The bijection is defined from chambers (represented
by permutations of $[n]$ decorated with arcs following certain rules)
to parking functions. See \cite{GMV} for a very general perspective of
this bijection. \emph{Centers} occur in the opposite direction, i.e.,
it is from the \emph{center} of any parking function that we may
recover the chamber labeled by it in the Pak-Stanley labeling
\cite{DGO}.

For example, consider the region $\cR$ of the Shi arrangement in $\Reals^9$
defined by
{\small \begin{align*}
& x_8<x_4<x_3<x_9<x_6<x_7<x_1<x_2<x_5\,,\\
& x_8+1>x_7,\,x_3+1>x_2,\, x_7+1>x_5\,,\\
& x_4+1<x_1,\, x_6+1<x_5\,.
\end{align*}

\noindent
According to this, we represent $\cR$ by\\[-20pt]
\begin{align*}
&\begin{tikzpicture}[scale=.3]
\draw (0,0) node [inner sep=.5mm,minimum size=2mm] (x8){$8$};
\draw (1,0) node [inner sep=.5mm,minimum size=2mm] (x4){$4$};
\draw (2,0) node [inner sep=.5mm,minimum size=2mm] (x3){$3$};
\draw (3,0) node [inner sep=.5mm,minimum size=2mm] (x9){$9$};
\draw (4,0) node [inner sep=.5mm,minimum size=2mm] (x6){$6$};
\draw (5,0) node [inner sep=.5mm,minimum size=2mm] (x7){$7$};
\draw (6,0) node [inner sep=.5mm,minimum size=2mm] (x1){$1$};
\draw (7,0) node [inner sep=.5mm,minimum size=2mm] (x2){$2$};
\draw (8,0) node [inner sep=.5mm,minimum size=2mm] (x5){$5$};
\draw  [thick,looseness=0.75]  (x8) to [out=90,in=90] (x7);
\draw  [thick,looseness=0.75]  (x7) to [out=90,in=90] (x5);
\draw  [thick,looseness=0.75]  (x3) to [out=90,in=90] (x2);
\end{tikzpicture}
\shortintertext{By the Pak-Stanley bijection, $\cR$ is associated with
the parking function}
&\ba=341183414\,.
\end{align*}

To obtain $\cR$ from $\ba$, note that the center of this parking function, $Z=\{3,4,6,7,8,9\}$, is formed by the elements of the first arc,
and their positions in the permutation $\pi=8\,4\,3\,9\,6\,7\,1\,2\,5$ can be determined by knowing that $a_i-1$ is the number of  elements less than $i$ that are on the left side of $i$ on $\pi$. Hence, we know how the above representation of $\cR$ begins.
By knowing this, in general (where a similar situation occurs),
we may replace the parking function by another one of shorter length,
and proceed recursively \cite{DGO}.}

\medskip

We now consider a third statistic.  Let, for
$\ba=(a_1,\dotsc,a_n)\in[n]^n$ such that $1 \in \{ a_1, \ldots, a_n \}$,
\begin{gather*}
\run(\ba)=\max\big\{i\in[n]\mid [i]\subseteq\{a_1,\dotsc,a_n\}\big\},
\shortintertext{and let $\run(\ba)=0$ if $1 \notin \{ a_1, \ldots, a_n
  \}$. We prove the following result in Section~\ref{bijection}. Let}
\RP_n(t)=\sum_{\ba\in\PF_n}t^{\run(\ba)}.
\end{gather*}
\begin{thmn}[\ref{pf2pf}]
For every $n\in\N$,
$$\ZP_n(t)=\RP_n(t)\,.$$
\end{thmn}
\medskip

Now, consider the set $\RW_n$ of \emph{rook words of length $n$} defined
by Leven, Rhoades and Wilson \cite{LRW}, that is, the ordered sets
$\ba=(a_1,\dotsc,a_n)\in[n]^n$ such that $a_1\leq\run(\ba)$.  Let
$$\RR_n(t)=\sum_{\ba\in\RW_n} t^{\run(\ba)}$$
The key to our enumeration is developed in Section~\ref{sequences},
where we prove the following result.
\begin{thmn}[\ref{pf2rw}]
For every $n\in\N$,
$$\RP_n(t)=\RR_n(t)\,.$$
\end{thmn}
In this case we do not consider all parking functions at
once. Instead, we only consider those for which the sets of elements
with the same image are the same. We count parking functions by
counting nonnegative sequences that are componentwise bounded above by
a given positive sequence. Based on  results of
  independent interest we prove that their number is the number of
  rook words defined in the same way.
\medskip

Finally, by directly counting rook words with a given $\run$, we are
able to present in Section~\ref{final} an expression for all the
(equal) previous enumerators.
\begin{thmn}[\ref{Cfinal}]
For every integers $1\leq r\leq n$,
\begin{align*}
[t^r]\big(\LT_n(t)\big)&= r! \sum_{i_1 + \dotsb + i_r = n-r}
(n-1)^{i_1} (n-2)^{i_2} \dotsb (n-r)^{i_r}\\ &= r\ \sum_{j=0}^{r-1}
(-1)^j \binom{r-1}{j} (n-1-j)^{n-1}\,.
\end{align*}
\end{thmn}

It is perhaps worth noting that rook words were introduced in order to
label the chambers of the \emph{Ish arrangement}, defined in
$\Reals^n$ by all the hyperplanes with equations of form $x_i-x_j=0$,
as before, and of form $x_1-x_j=i$, where again $1\leq i<j\leq n$.
Several bijections, which preserve different properties, have been
defined between the chambers of the Shi arrangement and the chambers
of the Ish arrangement, particularly by Leven, Rhoades and Wilson
using rook words \cite{LRW}.  In fact, our work here may be presented
as another example of a general statement by Armstrong and Rhoades
\cite{AR}, saying that ``The Ish arrangement is something of a `toy
model' for the Shi arrangement'', in the sense that several properties
are shared by both arrangements, but are easier to prove in case of
the Ish arrangement than in the case of the Shi arrangement.

\begin{example}
We consider in Table~\ref{tbl} the case where $n=3$ and hence
$$\LT_3(t)=\ZP_3(t)=\RP_3(t)=\RR_3(t)=4t+6t^2+6t^3$$ by classifying
the corresponding trees and parking functions  according to
the various statistics and the corresponding bijections.
Note that for $n=3$ rook words are parking functions and vice-versa, except that
$\text{3}\text{1}\text{1}\in\PF_3\setminus\RW_3$
whereas $\text{1}\text{3}\text{3}\in\RW_3\setminus\PF_3$. But
$\run(311)=\run(133)=1$.
\end{example}

\begin{table}[ht]\label{tbl}
$$\begin{array}{|c|c|c|c|}
\hline
k&1&2&3\\
\hline
\text{\begin{tabular}{c}Trees $T$\\with $\leg(T)=k$\end{tabular}}&
\begin{tikzpicture}[scale=0.55]
\coordinate [label=right:{\tiny$0$}] (n0) at (0., 0.);
\draw [fill] (n0) circle [radius=.1];
\coordinate [label=above:{\tiny$1$}] (n1) at (0., 1.);
\draw [fill] (n1) circle [radius=.1];
\coordinate [label=below:{\tiny$2$}] (n2) at (0.86, -0.5);
\draw [fill] (n2) circle [radius=.1];
\coordinate [label=below:{\tiny$3$}] (n3) at (-0.86, -0.5);
\draw [fill] (n3) circle [radius=.1];
\draw[ultra thick] (n0) -- (n2) ;
\draw[thick] (n0) -- (n1) ;
\draw[thick] (n1) -- (n3) ;
\end{tikzpicture}\,
\begin{tikzpicture}[scale=0.55]
\coordinate [label=right:{\tiny$0$}] (n0) at (0., 0.);
\draw [fill] (n0) circle [radius=.1];
\coordinate [label=above:{\tiny$1$}] (n1) at (0., 1.);
\draw [fill] (n1) circle [radius=.1];
\coordinate [label=below:{\tiny$2$}] (n2) at (0.86, -0.5);
\draw [fill] (n2) circle [radius=.1];
\coordinate [label=below:{\tiny$3$}] (n3) at (-0.86, -0.5);
\draw [fill] (n3) circle [radius=.1];
\draw[ultra thick] (n0) -- (n3) ;
\draw[thick] (n0) -- (n2) ;
\draw[thick] (n1) -- (n2) ;
\end{tikzpicture}&
\begin{tikzpicture}[scale=0.55]
\coordinate [label=right:{\tiny$0$}] (n0) at (0., 0.);
\draw [fill] (n0) circle [radius=.1];
\coordinate [label=above:{\tiny$1$}] (n1) at (0., 1.);
\draw [fill] (n1) circle [radius=.1];
\coordinate [label=below:{\tiny$2$}] (n2) at (0.86, -0.5);
\draw [fill] (n2) circle [radius=.1];
\coordinate [label=below:{\tiny$3$}] (n3) at (-0.86, -0.5);
\draw [fill] (n3) circle [radius=.1];
\draw[ultra thick] (n0) -- (n3) ;
\draw[ultra thick] (n1) -- (n3) ;
\draw[thick] (n0) -- (n2) ;
\end{tikzpicture}\,
\begin{tikzpicture}[scale=0.55]
\coordinate [label=right:{\tiny$0$}] (n0) at (0., 0.);
\draw [fill] (n0) circle [radius=.1];
\coordinate [label=above:{\tiny$1$}] (n1) at (0., 1.);
\draw [fill] (n1) circle [radius=.1];
\coordinate [label=below:{\tiny$2$}] (n2) at (0.86, -0.5);
\draw [fill] (n2) circle [radius=.1];
\coordinate [label=below:{\tiny$3$}] (n3) at (-0.86, -0.5);
\draw [fill] (n3) circle [radius=.1];
\draw[ultra thick] (n0) -- (n1) ;
\draw[ultra thick] (n1) -- (n3) ;
\draw[thick] (n1) -- (n2) ;
\end{tikzpicture}\,
\begin{tikzpicture}[scale=0.55]
\coordinate [label=right:{\tiny$0$}] (n0) at (0., 0.);
\draw [fill] (n0) circle [radius=.1];
\coordinate [label=above:{\tiny$1$}] (n1) at (0., 1.);
\draw [fill] (n1) circle [radius=.1];
\coordinate [label=below:{\tiny$2$}] (n2) at (0.86, -0.5);
\draw [fill] (n2) circle [radius=.1];
\coordinate [label=below:{\tiny$3$}] (n3) at (-0.86, -0.5);
\draw [fill] (n3) circle [radius=.1];
\draw[ultra thick] (n0) -- (n3) ;
\draw[ultra thick] (n3) -- (n2) ;
\draw[thick] (n1) -- (n3) ;
\end{tikzpicture}&
\begin{tikzpicture}[scale=0.55]
\coordinate [label=right:{\tiny$0$}] (n0) at (0., 0.);
\draw [fill] (n0) circle [radius=.1];
\coordinate [label=above:{\tiny$1$}] (n1) at (0., 1.);
\draw [fill] (n1) circle [radius=.1];
\coordinate [label=below:{\tiny$2$}] (n2) at (0.86, -0.5);
\draw [fill] (n2) circle [radius=.1];
\coordinate [label=below:{\tiny$3$}] (n3) at (-0.86, -0.5);
\draw [fill] (n3) circle [radius=.1];
\draw[ultra thick] (n0) -- (n3) ;
\draw[ultra thick] (n3) -- (n2) ;
\draw[ultra thick] (n2) -- (n1) ;
\end{tikzpicture}\,
\begin{tikzpicture}[scale=0.55]
\coordinate [label=right:{\tiny$0$}] (n0) at (0., 0.);
\draw [fill] (n0) circle [radius=.1];
\coordinate [label=above:{\tiny$1$}] (n1) at (0., 1.);
\draw [fill] (n1) circle [radius=.1];
\coordinate [label=below:{\tiny$2$}] (n2) at (0.86, -0.5);
\draw [fill] (n2) circle [radius=.1];
\coordinate [label=below:{\tiny$3$}] (n3) at (-0.86, -0.5);
\draw [fill] (n3) circle [radius=.1];
\draw[ultra thick] (n1) -- (n3) ;
\draw[ultra thick] (n3) -- (n2) ;
\draw[ultra thick] (n0) -- (n2) ;
\end{tikzpicture}\,
\begin{tikzpicture}[scale=0.55]
\coordinate [label=right:{\tiny$0$}] (n0) at (0., 0.);
\draw [fill] (n0) circle [radius=.1];
\coordinate [label=above:{\tiny$1$}] (n1) at (0., 1.);
\draw [fill] (n1) circle [radius=.1];
\coordinate [label=below:{\tiny$2$}] (n2) at (0.86, -0.5);
\draw [fill] (n2) circle [radius=.1];
\coordinate [label=below:{\tiny$3$}] (n3) at (-0.86, -0.5);
\draw [fill] (n3) circle [radius=.1];
\draw[ultra thick] (n1) -- (n3) ;
\draw[ultra thick] (n1) -- (n2) ;
\draw[ultra thick] (n0) -- (n2) ;
\end{tikzpicture}\\[-10pt]
&\begin{tikzpicture}[scale=0.55]
\coordinate [label=right:{\tiny$0$}] (n0) at (0., 0.);
\draw [fill] (n0) circle [radius=.1];
\coordinate [label=above:{\tiny$1$}] (n1) at (0., 1.);
\draw [fill] (n1) circle [radius=.1];
\coordinate [label=below:{\tiny$2$}] (n2) at (0.86, -0.5);
\draw [fill] (n2) circle [radius=.1];
\coordinate [label=below:{\tiny$3$}] (n3) at (-0.86, -0.5);
\draw [fill] (n3) circle [radius=.1];
\draw[ultra thick] (n0) -- (n3) ;
\draw[thick] (n0) -- (n1) ;
\draw[thick] (n1) -- (n2) ;
\end{tikzpicture}\,
\begin{tikzpicture}[scale=0.55]
\coordinate [label=right:{\tiny$0$}] (n0) at (0., 0.);
\draw [fill] (n0) circle [radius=.1];
\coordinate [label=above:{\tiny$1$}] (n1) at (0., 1.);
\draw [fill] (n1) circle [radius=.1];
\coordinate [label=below:{\tiny$2$}] (n2) at (0.86, -0.5);
\draw [fill] (n2) circle [radius=.1];
\coordinate [label=below:{\tiny$3$}] (n3) at (-0.86, -0.5);
\draw [fill] (n3) circle [radius=.1];
\draw[ultra thick] (n0) -- (n3) ;
\draw[thick] (n0) -- (n1) ;
\draw[thick] (n0) -- (n2) ;
\end{tikzpicture}&
\begin{tikzpicture}[scale=0.55]
\coordinate [label=right:{\tiny$0$}] (n0) at (0., 0.);
\draw [fill] (n0) circle [radius=.1];
\coordinate [label=above:{\tiny$1$}] (n1) at (0., 1.);
\draw [fill] (n1) circle [radius=.1];
\coordinate [label=below:{\tiny$2$}] (n2) at (0.86, -0.5);
\draw [fill] (n2) circle [radius=.1];
\coordinate [label=below:{\tiny$3$}] (n3) at (-0.86, -0.5);
\draw [fill] (n3) circle [radius=.1];
\draw[ultra thick] (n0) -- (n2) ;
\draw[ultra thick] (n2) -- (n3) ;
\draw[thick] (n1) -- (n2) ;
\end{tikzpicture}\,
\begin{tikzpicture}[scale=0.55]
\coordinate [label=right:{\tiny$0$}] (n0) at (0., 0.);
\draw [fill] (n0) circle [radius=.1];
\coordinate [label=above:{\tiny$1$}] (n1) at (0., 1.);
\draw [fill] (n1) circle [radius=.1];
\coordinate [label=below:{\tiny$2$}] (n2) at (0.86, -0.5);
\draw [fill] (n2) circle [radius=.1];
\coordinate [label=below:{\tiny$3$}] (n3) at (-0.86, -0.5);
\draw [fill] (n3) circle [radius=.1];
\draw[ultra thick] (n0) -- (n3) ;
\draw[ultra thick] (n3) -- (n2) ;
\draw[thick] (n0) -- (n1) ;
\end{tikzpicture}\,
\begin{tikzpicture}[scale=0.55]
\coordinate [label=right:{\tiny$0$}] (n0) at (0., 0.);
\draw [fill] (n0) circle [radius=.1];
\coordinate [label=above:{\tiny$1$}] (n1) at (0., 1.);
\draw [fill] (n1) circle [radius=.1];
\coordinate [label=below:{\tiny$2$}] (n2) at (0.86, -0.5);
\draw [fill] (n2) circle [radius=.1];
\coordinate [label=below:{\tiny$3$}] (n3) at (-0.86, -0.5);
\draw [fill] (n3) circle [radius=.1];
\draw[ultra thick] (n0) -- (n2) ;
\draw[ultra thick] (n3) -- (n2) ;
\draw[thick] (n0) -- (n1) ;
\end{tikzpicture}&
\begin{tikzpicture}[scale=0.55]
\coordinate [label=right:{\tiny$0$}] (n0) at (0., 0.);
\draw [fill] (n0) circle [radius=.1];
\coordinate [label=above:{\tiny$1$}] (n1) at (0., 1.);
\draw [fill] (n1) circle [radius=.1];
\coordinate [label=below:{\tiny$2$}] (n2) at (0.86, -0.5);
\draw [fill] (n2) circle [radius=.1];
\coordinate [label=below:{\tiny$3$}] (n3) at (-0.86, -0.5);
\draw [fill] (n3) circle [radius=.1];
\draw[ultra thick] (n0) -- (n3) ;
\draw[ultra thick] (n3) -- (n1) ;
\draw[ultra thick] (n1) -- (n2) ;
\end{tikzpicture}\,
\begin{tikzpicture}[scale=0.55]
\coordinate [label=right:{\tiny$0$}] (n0) at (0., 0.);
\draw [fill] (n0) circle [radius=.1];
\coordinate [label=above:{\tiny$1$}] (n1) at (0., 1.);
\draw [fill] (n1) circle [radius=.1];
\coordinate [label=below:{\tiny$2$}] (n2) at (0.86, -0.5);
\draw [fill] (n2) circle [radius=.1];
\coordinate [label=below:{\tiny$3$}] (n3) at (-0.86, -0.5);
\draw [fill] (n3) circle [radius=.1];
\draw[ultra thick] (n0) -- (n1) ;
\draw[ultra thick] (n1) -- (n3) ;
\draw[ultra thick] (n3) -- (n2) ;
\end{tikzpicture}\,
\begin{tikzpicture}[scale=0.55]
\coordinate [label=right:{\tiny$0$}] (n0) at (0., 0.);
\draw [fill] (n0) circle [radius=.1];
\coordinate [label=above:{\tiny$1$}] (n1) at (0., 1.);
\draw [fill] (n1) circle [radius=.1];
\coordinate [label=below:{\tiny$2$}] (n2) at (0.86, -0.5);
\draw [fill] (n2) circle [radius=.1];
\coordinate [label=below:{\tiny$3$}] (n3) at (-0.86, -0.5);
\draw [fill] (n3) circle [radius=.1];
\draw[ultra thick] (n0) -- (n1) ;
\draw[ultra thick] (n1) -- (n2) ;
\draw[ultra thick] (n2) -- (n3) ;
\end{tikzpicture}\\
\hline
\text{Parking functions $\ba$}&
\text{2}\textbf{\underline1}\text{3}\ \quad
\text{2}\text{2}\textbf{\underline1}&
\textbf{\underline1}\text{3}\textbf{\underline1}\ \quad
\textbf{\underline1}\text{3}\textbf{\underline2}\ \quad
\text{2}\textbf{\underline1}\textbf{\underline1}&
\textbf{\underline1}\textbf{\underline1}\textbf{\underline1}\ \quad
\textbf{\underline1}\textbf{\underline1}\textbf{\underline2}\ \quad
\textbf{\underline1}\textbf{\underline1}\textbf{\underline3}\\
\text{with $z(\ba)=k$}&\text{2}\text{3}\textbf{\underline1}\ \quad
\text{3}\text{2}\textbf{\underline1}&
\text{2}\textbf{\underline1}\textbf{\underline2}\ \quad
\text{3}\textbf{\underline1}\textbf{\underline1}\ \quad
\text{3}\textbf{\underline1}\textbf{\underline2}&
\textbf{\underline1}\textbf{\underline2}\textbf{\underline1}\ \quad
\textbf{\underline1}\textbf{\underline2}\textbf{\underline2}\ \quad
\textbf{\underline1}\textbf{\underline2}\textbf{\underline3}\\
\hline
\text{Parking functions $\ba$}&
\text{1}\textbf{\underline1}\text{3}\ \quad
\text{1}\text{1}\textbf{\underline1}&
\textbf{\underline2}\text{1}\textbf{\underline1}\ \quad
\textbf{\underline1}\text{2}\textbf{\underline2}\ \quad
\text{2}\textbf{\underline2}\textbf{\underline1}&
\textbf{\underline3}\textbf{\underline2}\textbf{\underline1}\ \quad
\textbf{\underline2}\textbf{\underline3}\textbf{\underline1}\ \quad
\textbf{\underline2}\textbf{\underline1}\textbf{\underline3}\\
\text{with $\run(\ba)=k$}&
\text{1}\text{3}\textbf{\underline1}\ \quad
\text{3}\text{1}\textbf{\underline1}&
\text{1}\textbf{\underline1}\textbf{\underline2}\ \quad
\text{1}\textbf{\underline2}\textbf{\underline1}\ \quad
\text{2}\textbf{\underline1}\textbf{\underline2}&
\textbf{\underline3}\textbf{\underline1}\textbf{\underline2}\ \quad
\textbf{\underline1}\textbf{\underline3}\textbf{\underline2}\ \quad
\textbf{\underline1}\textbf{\underline2}\textbf{\underline3}\\
\hline
\end{array}$$
\caption{The case where $n=3$}
\end{table}

\section{From labeled trees to parking functions: legs \emph{vs.} centers}\label{algorithms}

We reproduce here the algorithm (Algorithm~\ref{dfs-burning}, below)
of Perkinson, Yang and Yu \cite{PYY}, that takes as input the parking function
$\ba$ (or, more precisely, takes as input $\pf = \ba-1 \colon [n] \to \N\cup\{0\}$ such that $\pf(i)=a_i-1$) and returns
the list $\tree$ of edges of a tree. Note that a parking function is a
$G$-parking function where $G$ is a (rooted) complete graph on
$V$. For this reason, the sentence ``for each $j$ adjacent to $i$ in
$G$'' reads as ``for each $j\neq i$''.

Note also that, in general, a spanning tree $T$ of $G$ is seen as a
directed graph in which all paths lead away from the root. So, edge
$ij$ is written $(i,j)$ if $i$ is in the (unique) path from $i$ to $j$
(with no vertex between them).  Note also that, by definition
(cf. Line~7), if after running the algorithm both edges $(i,j)$ and
$(i,k)$ belong to $T$ and $j>k$ then {\sc dfs\_burn}{$(j)$} has been
called before {\sc dfs\_burn}{$(k)$}.

{\small
\begin{algorithm}[ht]
\caption{ {\bf $\dfs$-burning algorithm}.}\label{dfs-burning}
\begin{algorithmic}[1]
  \Statex
  \Statex{\hspace{-0.5cm}\sc algorithm}
  \Statex{\bf Input:} $\pf\colon V\setminus\{r\}\to\N\cup\{0\}$
  \State $\bnv =\{r\}$
  \State $\damp = \{\,\}$
  \State $\tree= \{\,\}$
  \State execute {\sc dfs\_from}($r$)
  \Statex{\bf Output:} {$\bnv$ and $\tree$}
  \Statex \rule{0.5\textwidth}{.4pt}
  \Statex{\hspace{-0.5cm}\sc auxiliary function}
    \Function{\sc dfs\_from}{$i$}
      \ForAll{$j$ adjacent to $i$ in $G$, from largest numerical
              value to smallest}
	\If{$j\notin\bnv$}
	  \If{$\pf(j)=0$}
	    \State append $j$ to $\bnv$
	    \State append $(i,j)$ to $\tree$
	    \State {\sc dfs\_from}$(j)$
	  \Else
	     \State $\pf(j) = \pf(j)-1$
	     \State append $(i,j)$ to $\damp$
	  \EndIf
       \EndIf	
     \EndFor
   \EndFunction
  \end{algorithmic}
\end{algorithm}
}

\begin{theorem}\label{t2pf}
For every $n\in\N$,
\begin{equation*}\LT_n(t)=\ZP_n(t)\end{equation*}
\end{theorem}
\begin{proof}
We show that there exist a bijection $\varphi\colon\PF_n\to\T_n$ such
that, for every $\ba\in\PF_n$, if $T=\varphi(\ba)\in\T_n$, then
$\leg(T)=z(\ba)$.

Let $T$ be the tree given by Algorithm~\ref{dfs-burning} with input
$\pf=\ba-1$ (we know this defines a bijection from $\PF_n$ to $\T_n$
by \cite[Theorem 3]{PYY}).  Now, let $\ell$ be the first value of $i$
where, when {\sc dfs\_from$(i)$} is called, $\pf(j)>0$ whenever
$j\notin\bnv$. If this never occurs, let $\ell$ be the last vertex
joined to $\bnv$.

Let $B=\big(0\!\!=\!\!v_0,\dotsc,v_k\!\!=\!\!i\big)=\bnv$ and
$E=\tree$ at the end of the loop of {\sc dfs\_from$(i)$} (the end of
Line~14) for $i=\ell$, and note that, by definition,
$E=\big((v_0,v_1),\dotsc,(v_{k-1},v_k)\big)$.  Hence,
$\bv(T)=(v_1,\dotsc,v_k)$.

Now, let $X=\{x_1,\dotsc,x_k\}=\{v_1,\dotsc,v_k\}$ with
$x_1<\dotsb<x_k$.

We must prove that:
\begin{enumerate}
\item for every $m\in[k]$, $\ba(x_m)\leq m$;
\item $X$ is maximal for this property;
\end{enumerate}
Clearly, if $x_{i_1}=v_1$, then $\ba(x_{i_1})\leq i_1$ since, by
definition, $v_1=\max\big(\pf^{-1}(\{0\})\big)$ and so
$\ba(x_{i_1})=1$.  Now, suppose that the same holds true for
$x_{i_2}=v_2,\dotsc,x_{i_{\ell-1}}=v_{\ell-1}$, consider
$x_{i_\ell}=v_\ell$ and note that, when {\sc dfs\_from$(v_{\ell-1})$}
is called, $v_\ell$ is the largest value of
$j\notin\{v_0,v_1,\dotsc,v_{\ell-1}\}$ with $\pf(j)=0$. Since
$\pf(v_\ell)$ has been reduced in earlier calls to {\sc
  dfs\_from$(v_m)$} (at Line~13) exactly when $v_m<v_\ell$, since it
is now zero, and since new additions to $\bnv$ will not decrease the
order of $v_\ell$ in the corresponding set, $\ba(x_{i_\ell})\leq
i_\ell$.

When finally {\sc dfs\_from$(v_k)$} is called, $\pf(j)>0$ for all
$j\notin X$. In particular, for each such $j$, the number of elements
in $X$ less than or equal to $j$ is necessarily less than $a_j$.
\end{proof}
More precisely, if $\bv(T)=(v_1,\dotsc,v_k)$, then clearly
$$\ba (v_i)=\big|\big\{k\in[i]\mid v_k\leq v_i\big\}\big|\,.$$
Compare with Definition~\ref{deft} below.
\begin{example} \label{expf}
Let $\ba=341183414 \in \PF_9$. We apply Algorithm~\ref{dfs-burning} to
$\ba$ by drawing $a_j$ empty boxes for each $j\in[9]$ that are filled
with $i$ during the execution of {\sc dfs\_from}{($i$)}, at Line~14 and at Line~10.
Below, {\sc dfs\_from}{($i$)} has been called for, in this order,
$i=0,8,4,3,9,6,7$. Hence, at the moment, $i=7$ and
$\bnv=(0,8,4,3,9,6,7)$.  Since $\pf(j)>0$ for $j\notin\bnv$ (i.e., for
$j=1,2,5$), $\ell=i=7$, and so $\bv=(8,4,3,9,6,7)$.

{\small
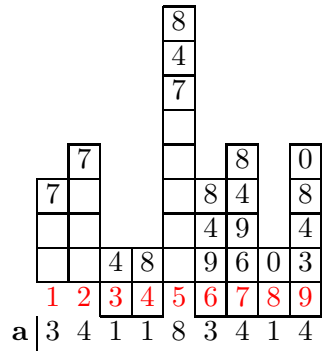
$$\begin{array}{rccccccccc}
\cline{6-6}
&&&&&\cx{8}&\\
\cline{6-6}
&&&&&\cx{4}&\\
\cline{6-6}
&&&&&\cx{7}&\\
\cline{6-6}
&&&&&\cx{}&\\
\cline{3-3}\cline{6-6}\cline{8-8}\cline{10-10}
&&\cx{7}&&&\cx{}&&\cx{8}&&\cx{0}\\
\cline{2-3}\cline{6-8}\cline{10-10}
&\cx{7}&\cx{}&&&\cx{}&\cx{8}&\cx{4}&&\cx{8}\\
\cline{2-3}\cline{6-8}\cline{10-10}
&\cx{}&\cx{}&&&\cx{}&\cx{4}&\cx{9}&&\cx{4}\\
\cline{2-10}
&\cx{}&\cx{}&\cx{4}&\cx{8}&\cx{}&\cx{9}&\cx{6}&\cx{0}&\cx{3}\\
\cline{2-10}
&\red\!1\!&\red\!2\!&\cxr{3}&\cxr{4}&\red\!5\!&\cxr{6}&\cxr{7}&\cxr{8}&\cxr{9}\\
\cline{4-5}\cline{7-10}
\mcx{\ba}&\!3\!&\!4\!&\!1\!&\!1\!&\!8\!&\!3\!&
\!4\!&\!1\!&\!4\!
\end{array}$$}

\end{example}

\section{Within parking functions: centers \emph{vs.} runs}\label{bijection}
\begin{definition}\label{deft}
Consider, for a positive integer $n$ and for a permutation $\bw = (w_1, \ldots, w_n) \in\mathfrak{S}_n$,
\begin{align*}
  &f_{w_i}=\left|\big\{k\in[i]\,\mid\, w_k\leq w_i\big\}\right|\,,\
  i=1,\dotsc,n\ ,
\shortintertext{and}
  &t_n (\bw)=(f_1,\dotsc,f_n)\in [1] \times [2] \times \cdots \times [n] \,.
\end{align*}
According to \cite{DGO}, $t_n$ is a bijection between $\mathfrak{S}_n$
and $[1] \times [2] \times \cdots \times [n]$.
\end{definition}

\begin{example}
If $\bw = 521634$, then $f_1 = f_{w_3} = 1$, $f_2 = f_{w_2} = 1$, $f_3
= f_{w_5} = 3$, $f_4 = f_{w_6} = 4$, $f_5 = f_{w_1} = 1$ and $f_6 =
f_{w_4} = 4$. Hence $t(\bw) = 113414 \in [1] \times \cdots \times
[6]$.
\end{example}

Given $\ba \in [n]^n$, let
$$ \label{defrun}
\Run (\ba) = \big\{ \max \ \ba^{-1} (\{ j \}) \mid 1 \leq j \leq \run (\ba) \big\}
$$
if $\run(\ba)>0$, and let $\Run (\ba) = \varnothing$ if $\run(\ba)=0$.
Then $|\Run (\ba)| = \run (\ba)$.

For $A \subseteq [n]$, let
$${Z_n}^{-1} (A) = \big\{ \ba \in [n]^n \mid Z(\ba)=A \big\}$$
and
$${\Run_n}^{-1} (A) = \big\{ \ba \in [n]^n \mid \Run (\ba)=A \big\}.$$

Now let $A = \{ i_1, i_2, \ldots, i_k \} \neq \varnothing$ with $i_1 <
i_2 < \cdots < i_k$ and take $i_0=0$ and $i_{k+1}=n+1$. If $\ba =
(a_1, a_2, \ldots, a_n) \in {Z_n}^{-1} (A)$, then
\begin{itemize}
\item $(a_{i_1},a_{i_2},\ldots,a_{i_k}) \in [1] \times [2] \times
  \cdots \times [k]$,
\item $a_j \in \{ \ell+1, \ldots, n \} = [n] \setminus [\ell]$, if
  $i_{\ell-1} < j < i_\ell$, with $\ell \in [k+1]$.
\end{itemize}
{ \small
$$
\begin{array}{ccccccccc}
{\scriptstyle \neq 1} & \fbox{$a_{i_1} {\scriptstyle \leq 1}$} &
{\scriptstyle \neq \begin{cases} \scriptstyle 1 \\[-5pt] \scriptstyle
    2 \end{cases}} & \fbox{$a_{i_2} {\scriptstyle \leq 2}$} &
{\scriptstyle \neq \begin{cases} \scriptstyle 1 \\[-5pt] \scriptstyle
    2 \\[-5pt] \scriptstyle 3 \end{cases}} & & {\neq \begin{cases}
    \scriptstyle 1 \\[-5pt] \vdots \\[-5pt] \scriptstyle
    k \end{cases}} & \fbox{$a_{i_k} {\scriptstyle \leq k}$} &
{\scriptstyle \neq \begin{cases} \scriptstyle 1 \\[-5pt] \vdots
    \\[-5pt] \scriptstyle k+1 \end{cases}}\\[-5pt]
 \underline{\hspace{1.2cm}} &
\underbracket[.5pt]{\hspace{.8cm}}_{i_1} & \underline{\hspace{1.2cm}}
& \underbracket[.5pt]{\hspace{.8cm}}_{i_2} &
\underline{\hspace{1.2cm}} & \ldots & \underline{\hspace{1.2cm}} &
\underbracket[.5pt]{\hspace{.8cm}}_{i_k} & \underline{\hspace{1.2cm}}
\\
\end{array}
$$}
If $\ba = (a_1, a_2, \ldots, a_n) \in {\Run_n}^{-1} (A)$, then
\begin{itemize}
\item $(a_{i_1},a_{i_2},\ldots,a_{i_k}) \in \mathfrak{S}_k$,
\item $a_j \in [n] \setminus \{ k+1, a_{i_1}, \ldots, a_{i_{\ell-1}}
  \}$, if $i_{\ell-1} < j < i_\ell$, with $\ell \in [k+1]$.
\end{itemize}
{\small
$$
\begin{array}{ccccccccc}
{\scriptstyle \neq k+1} & \fbox{$a_{i_1}$} & {\scriptstyle
  \neq \begin{cases} \scriptstyle k+1 \\[-5pt] \scriptstyle
    a_{i_1} \end{cases}} & \fbox{$a_{i_2}$} & {\scriptstyle
  \neq \begin{cases} \scriptstyle k+1 \\[-5pt] \scriptstyle a_{i_1}
    \\[-5pt] \scriptstyle a_{i_2} \end{cases}} & & {\neq \begin{cases}
    \scriptstyle k+1 \\[-5pt] \ \ \vdots \\[-5pt] \scriptstyle
    a_{i_{k-1}} \end{cases}} & \fbox{$a_{i_k}$} & {\scriptstyle
  \neq \begin{cases} \scriptstyle k+1 \\[-5pt] \ \vdots \\[-5pt]
    \scriptstyle a_{i_k} \end{cases}} \\[-5pt]
\underline{\hspace{1.2cm}} & \underbracket[.5pt]{\hspace{.8cm}}_{i_1}
& \underline{\hspace{1.2cm}} &
\underbracket[.5pt]{\hspace{.8cm}}_{i_2} & \underline{\hspace{1.2cm}}
& \ldots & \underline{\hspace{1.2cm}} &
\underbracket[.5pt]{\hspace{.8cm}}_{i_k} & \underline{\hspace{1.2cm}}
\\
\end{array}
$$}
Clearly, the size of both ${Z_n}^{-1} (A)$ and ${\Run_n}^{-1} (A)$ is
$$ k! (n-1)^{i_1-1} (n-2)^{i_2-1} \cdots (n-k)^{i_k-1}$$
if $|A| = k > 0$, and $(n-1)^n$ if $A = \varnothing$.

We now define two mappings $\Phi, \Psi: [n]^n \to [n]^n$.

If $Z(\ba) = \varnothing$, we define $\Phi(\ba) := \ba$. Otherwise, if
$Z(\ba) = \{ i_1, i_2, \ldots, i_k \} \neq \varnothing$ with $i_1 <
i_2 < \cdots < i_k$, we define $\Phi(\ba)$ as follows. Let $\bb := b_1
b_2 \ldots b_k = {t_k}^{-1} (a_{i_1} a_{i_2} \ldots a_{i_k}) \in
\mathfrak{S}_k$ and $\sigma_\ba \in \mathfrak{S}_n$ be the permutation
of length $n$ defined by
$$
\sigma_\ba (j) =
\begin{cases}
k+1, & \text{if $j=1$;} \\
b_{j-1}, & \text{if $2 \leq j \leq k+1$;} \\
j, & \text{if $k+2 \leq j \leq n$.}
\end{cases}
$$

Finally, let
$$
(\Phi(\ba))(j) :=
\begin{cases}
b_\ell, & \text{if $j = i_\ell \in Z(\ba)$;} \\
\sigma_\ba (a_j), & \text{if $j \notin Z(\ba)$.}
\end{cases}
$$

If $\Run(\ba) = \varnothing$, we define $\Psi(\ba) := \ba$. Otherwise, if $\Run(\ba) = \{ i_1, i_2, \ldots, i_k \} \neq \varnothing$ with $i_1 < i_2 < \cdots < i_k$, we define $\Psi(\ba)$ as follows. Let $\bc := c_1 c_2 \ldots c_k = t_k (a_{i_1} a_{i_2} \ldots a_{i_k}) \in [1] \times [2] \times \cdots \times [k]$ and $\tau_\ba \in \mathfrak{S}_n$ be the permutation of length $n$ defined by
$$
\tau_\ba (j) =
\begin{cases}
\ell+1, & \text{if $j = a_{i_\ell} \in [k]$;} \\
1, & \text{if $j=k+1$;} \\
j, & \text{if $k+2 \leq j \leq n$.}
\end{cases}
$$

Finally, let
$$
(\Psi(\ba))(j) :=
\begin{cases}
c_\ell, & \text{if $j = i_\ell \in \Run(\ba)$;} \\
\tau_\ba (a_j), & \text{if $j \notin \Run(\ba)$.}
\end{cases}
$$

\begin{example}
Let $\ba = 341183414 \in [9]^9$. On the one hand, $Z(\ba) = \{
3,4,6,7,8,9 \}$, ${t_6}^{-1} (a_3 a_4 a_6 a_7 a_8 a_9) = {t_6}^{-1}
(113414) = 521634 \in \mathfrak{S}_6$, so $\sigma_\ba =
7\underline{521634}89 $ and $\Phi(\ba) =
21\underline{52}8\underline{1634}$.
{\small
$$
\begin{array}{c}
\begin{array}{cccccccccccc}
\ba & = & 3 & 4 & \fbox{$1$} & \fbox{$1$} & 8 & \fbox{$3$} &
\fbox{$4$} & \fbox{$1$} & \fbox{$4$} \\[-5pt] &&
\multicolumn{2}{l}{\!\! \underline{\hspace{1cm}}} &
\underbracket[.5pt]{\hspace{.6cm}}_{3} &
\underbracket[.5pt]{\hspace{.6cm}}_{4} & \underline{\hspace{.6cm}} &
\underbracket[.5pt]{\hspace{.6cm}}_{6} &
\underbracket[.5pt]{\hspace{.6cm}}_{7}&
\underbracket[.5pt]{\hspace{.6cm}}_{8}&
\underbracket[.5pt]{\hspace{.6cm}}_{9} \\ \\ \\
\Phi(\ba) & = & 2 & 1 & \fbox{$5$} & \fbox{$2$} & 8 & \fbox{$1$} &
\fbox{$6$} & \fbox{$3$} & \fbox{$4$} \\[-5pt] &&
\multicolumn{2}{l}{\!\!  \underline{\hspace{1cm}}} &
\underbracket[.5pt]{\hspace{.6cm}}_{3} &
\underbracket[.5pt]{\hspace{.6cm}}_{4} & \underline{\hspace{.6cm}} &
\underbracket[.5pt]{\hspace{.6cm}}_{6} &
\underbracket[.5pt]{\hspace{.6cm}}_{7}&
\underbracket[.5pt]{\hspace{.6cm}}_{8}&
\underbracket[.5pt]{\hspace{.6cm}}_{9} \\
\end{array}
\end{array}
$$}
On the other hand, $\Run(\ba) = \{ 8 \}$, $\bc = t_1 (a_8) = t_1 (1) =
1$, so $\tau_\ba = 2\underline{1}3456789$ and
$\Psi(\ba)=3422834\underline{1}4$.  Note that $\ba$ belongs to
$\PF_9$, as well as $\Phi(\ba)$ and $\Psi(\ba)$.
\end{example}

The following theorem summarizes the main properties of $\Phi$ and $\Psi$.

\begin{theorem} \hfill
\begin{enumerate}
\item For all $\ba \in [n]^n$, $z(\ba) = \run (\Phi (\ba))$, $Z(\ba) =
  \Run (\Phi (\ba))$, $\run (\ba) = z(\Psi (\ba))$, and $\Run (\ba) =
  Z(\Psi (\ba))$,
\item $\Phi$ and $\Psi$ are bijections and $\Psi = \Phi^{-1}$,
\item $\Phi (\PF_n) = \PF_n$.
\end{enumerate}
\end{theorem}

\begin{proof} \hfill

(1) Let $Z(\ba)=\{i_1,\dotsc,i_k\}$ with $i_1<\dotsb<i_k$ and
  $\Phi(\ba)=:\bd=d_1\dotsb d_n$. On the one hand, $k \leq \run (\bd)$
  since $\{d_{z_1},\dotsc, d_{z_k}\}= [k]$. On the other hand,
  $k+1\notin\{d_1,\dotsc, d_n\}$ because $k+1 \notin \{d_{z_1},\dotsc,
  d_{z_k}\}$ and there is no $j \in [n] \setminus Z(\ba)$ such that
  $\ba(j) = 1$. Hence $z(\ba) =k= \run (\bd)$.

Let $i_j \in Z(\ba)$.  First, note
that $d_{i_j} \leq j \leq k$. Second, note that if $\ell >
i_j$, then $d_\ell\neq d_{i_j}$; if $\ell = i_p \in
Z(\ba) \setminus \{ i_j \}$, then $d_\ell \neq
d_{i_j}$ because $d_{z_1}\dotsb d_{z_k}\in \mathfrak{S}_k$;
if $\ell \in [n] \setminus \left(
    [i_j] \cup Z(\ba) \right)$, then $\ba(\ell) > j + 2$ and
    $d_\ell \neq d_{i_j}$, since $\sigma_\ba \in
    \mathfrak{S}_n$.

Similarly, one can show that $\run (\ba) = z(\Psi (\ba))$ and $\Run
(\ba) = Z(\Psi (\ba))$.

(2) Given $\ba \in [n]^n$, we have $Z(\ba) = \Run (\Phi (\ba))$,
  $\Run (\ba) = Z(\Psi (\ba))$, $\tau_{\Phi(\ba)} = {\sigma_\ba}^{-1}$
  and $\sigma_{\Psi(\ba)} = {\tau_\ba}^{-1}$. Hence $(\Psi \circ \Phi)
  (\ba) = \ba = (\Phi \circ \Psi) (\ba)$.

(3) Let $\ba \in \PF_n$ and $k=z(\ba)$. If $j \leq k$,
  $|\Phi (\ba)^{-1} ([j])| \geq j$, since $[j] \subseteq [k] \subseteq
  \Phi (\ba)([n])$. If $j > k$, then $\Phi (\ba)^{-1} ([j]) = \ba^{-1}
  ([j])$ and so $|\Phi (\ba)^{-1} ([j])| = |\ba^{-1} ([j])| \geq j$
  because $\ba \in \PF_n$. Since $\Phi (\PF_n) \subseteq \PF_n$ and
  $\Phi$ is a bijection, $\Phi (\PF_n) = \PF_n$. \qedhere
\end{proof}

\begin{theorem}\label{pf2pf}
For every $n\in\N$,
$$\ZP_n(t)=\RP_n(t)\,.$$
\end{theorem}

\section{From parking functions to rook words}\label{sequences}

\subsection{Restricted integer sequences}

We start this section by considering a general situation of
independent interest.

\begin{definition}\label{def:sbr}
Let, for a positive integer $k$ and for
$\bl=(\ell_1,\dotsc,\ell_k)\in\N^k$, $\bL=(L_1,\dotsc,L_k)\in\N^k$ be
the cumulative sum of $\bl$, i.e.,
\begin{gather*}
L_i=\ell_1+\ell_2+\dotsb+\ell_i\,,\quad i=1,\dotsc,k\,,
\shortintertext{and consider the set}
\sbr{\ell_1,\dotsc,\ell_k}=\big\{(x_0,x_1,\dotsc,x_k)\in\Z^{k+1}\mid
x_0=0;\forall 1\leq i\leq k\,, x_{i-1}<x_i\leq L_i\big\}
\end{gather*}
\end{definition}

\begin{lemma}\label{lem:main} For every positive integers
$k,\ell_1,\dotsc,\ell_k$, if $i<k$ and $\ell_{i+1}>1$, then
\begin{align*}
&\nbr{\ell_1,\dotsc,\ell_{i-1},\ell_i+1,\ell_{i+1}-1,\ell_{i+2},\dotsc,\ell_k}\\
&\hphantom{\nbr{\ell_1,\dotsc,\ell_{k-1},\ell_k+1}}=
\nbr{\ell_1,\dotsc,\ell_k}+\nbr{\ell_1,\dotsc,\ell_{i-1}}
\nbr{\ell_{i+1}-1,\ell_{i+2},\dotsc,\ell_k}
\shortintertext{whereas}
&\nbr{\ell_1,\dotsc,\ell_{k-1},\ell_k+1}=
\nbr{\ell_1,\dotsc,\ell_k}+\nbr{\ell_1,\dotsc,\ell_{k-1}}
\shortintertext{and, if $\ell_1>1$,}
&\nbr{\ell_1-1,\ell_2,\dotsc,\ell_k}=\nbr{\ell_1,\dotsc,\ell_k}-
\nbr{\ell_1+\ell_2-1,\ell_3,\dotsc,\ell_k}\,.
\end{align*}
\end{lemma}
\begin{proof}
We present here a bijective proof.  Note that, for every $i<k$,
\begin{align*}
&\sbr{\ell_1,\dotsc,\ell_k}\subseteq\sbr{\ell_1,\dotsc,\ell_{i-1},\ell_i+1,
\ell_{i+1}-1,\ell_{i+2},\dotsc,\ell_k}\,.\\
\shortintertext{But, by definition,}
&(x_0,\dotsc,x_k)\in\sbr{\ell_1,\dotsc,\ell_{i-1},\ell_i+1,\ell_{i+1}-1,\ell_{i+2},
  \dotsc,\ell_k}\setminus\sbr{\ell_1,\dotsc,\ell_k}\\
&\iff\left\{\begin{array}{l}
 x_0=0\\
 x_i=L_i+1\\
\ x_{j-1}<x_j\leq L_j\  \text{for every $j\neq i$ with $1\leq j\leq k$,}
\end{array}\right.\\
&\iff\left\{\begin{array}{l}
 (x_0,\dotsc,x_{i-1})\in\sbr{\ell_1,\dotsc,\ell_{i-1}}\\[5pt]
 (x_i-L_i-1,\dotsc,x_k-L_i-1)\in\sbr{\ell_{i+1},\dotsc,\ell_k}
\end{array}\right.
\end{align*}

For the second statement, note that also
$\sbr{\ell_1,\dotsc,\ell_k}\subseteq
\sbr{\ell_1,\dotsc,\ell_{k-1},\ell_k+1}$ and that
$(x_0,\dotsc,x_k)\in\sbr{\ell_1,\dotsc,\ell_{k-1},\ell_k+1}\setminus
\sbr{\ell_1,\dotsc,\ell_k}$ if and only if $x_k=\ell_k+1$ and
$(x_0,\dotsc,x_{k-1})\in\sbr{\ell_1,\dotsc,\ell_{k-1}}$.

Finally, for the third statement, note that, by definition, if
\begin{align*}
&(0,x_1,\dotsc,x_k)\in\sbr{\ell_1-1,\ell_2,\dotsc,\ell_k}\,,
\shortintertext{then}
&\left\{\begin{array}{l}
x_1+1>1\\[2.5pt]
(0,x_1+1,\dotsc,x_k+1)\in\sbr{\ell_1,\ell_2,\dotsc,\ell_k}\,.
\end{array}\right.
\shortintertext{In fact, given a $k$-tuple $(y_1,\dotsc,y_k)\in\N^k$,}
&\left\{\begin{array}{l}
(0,y_1,\dotsc,y_k)\in\sbr{\ell_1,\ell_2,\dotsc,\ell_k}\\[2.5pt]
(0,y_1-1,\dotsc,y_k-1)\notin\sbr{\ell_1-1,\ell_2,\dotsc,\ell_k}
\end{array}\right.
\shortintertext{if and only if}
&\left\{\begin{array}{l}
y_1=1\\[2.5pt]
(0,y_2-1,\dotsc,y_k-1)\in\sbr{\ell_1+\ell_2-1,\ell_3,\dotsc,\ell_k}.\qedhere
\end{array}\right.
\end{align*}
\end{proof}

\begin{remark}
Let, for any $\bx=(0,x_1,\dotsc,x_k)\in\Z^{k+1}$,
$\by=(y_1,\dotsc,y_k)=(x_1-1,x_2-2,\dotsc, x_k-k)$.  Then $\bx\in
\sbr{\ell_1,\dotsc,\ell_k}$ if and only if $0\leq y_1\leq L_1-1$ and
$y_i\leq y_{i+1}\leq L_{i+1}-(i+1)$ for every $i=1,2,\dotsc,k-1$.

Hence, if $(\ell_1,\dotsc,\ell_k)$ is a \emph{composition of $n$}
(i.e., $n=L_k$) we may represent the elements of
$\sbr{\ell_1,\dotsc,\ell_k}$ by lattice paths from $(0,0)$ to
$(k,n-k)$ that are contained in the region between the $x$ axis and
the path $P$ that has the same ends and the property that the height
of the $i$th horizontal step is $L_i-i$ for every $i=1,2,\dotsc,k$.
See Figure~\ref{fig1} for an example.  Hence,
\begin{equation}
\nbr{\ell_1,\dotsc,\ell_k}=\det\limits_{1\leq i,j\leq k}\Bigg(
\binom{\ell_1+\dotsb+\ell_i-i+1}{j-i+1}\Bigg)\,.
\label{eq:contas1}
\end{equation}
follows (cf. \cite[Theorem 10.7.1]{Krat}). Note that
Lemma~\ref{lem:main} may easily be proved by using the characteristic
properties of determinants.
\end{remark}

\begin{example}
Note that $(0,1,2)$ is a subsequence of $\bx=(0,1,2,7,9)$ but
$(0,1,2,3)$ is not. Let $S$ be the set of elements of $\sbr{3,1,5,2}$
with this property,
$$S=\{(0,1,2,x_3,x_4)\in\sbr{3,1,5,2}\mid x_3>3\}$$
and note that the lattice paths associated with the elements of $S$
are those which start by $2$ horizontal steps, followed by a vertical
step (cf. Figure~\ref{fig1}).  Now,
$(0,1,2,x_3,x_4)\mapsto(0,x_3-3,x_4-3)$ defines a bijection between
$S$ and $\sbr{(3+1+5)-3,2}=\sbr{6,2}$.

\begin{figure}[ht]
\centering
\begin{tikzpicture}[scale=.85]
\draw[->] (0,0) -- (4.3,0);
\draw[->] (0,0) -- (0,7.3);
\draw[very thick,gray, pattern color=gray,pattern=north east lines, fill] (4,0) -- (0,0) -- (0,2) -- (2,2) -- (2,6) to node[above] {$P_\downarrow$}   (3,6) -- (3,7) -- (4,7) ;
\tikzstyle{every node}=[circle,inner sep = 0.25pt, draw=white, fill=white]
\foreach \y in {1,2,3,4,5,6,7}
\draw[yshift=\y cm] (-5pt,0pt) -- (5pt,0pt) ;
\foreach \y in {0,1,2}
\foreach \x in {0,1,2,3,4}
\draw [fill] (\x , \y) circle[radius=0.05cm];
\foreach \y in {3,4,5,6}
\foreach \x in {2,3,4}
\draw [fill] (\x , \y) circle[radius=0.05cm];
\foreach \x in {3,4}
\draw [fill] (\x , 7) circle[radius=0.05cm];
\draw[very thick] (0,0) to node[above=1pt] {$0$} (1,0)  to node[above=1pt] {$0$} (2,0) -- (2,4)  to node[above=1pt] {$4$} (3,4) -- (3,5)
  to node[above=1pt] {$5$} (4,5) -- (4,7)  ;
\end{tikzpicture}
\caption{Lattice path representation of $(0,1,2,7,9)\in\sbr{3,1,5,2}$.}
\label{fig1}
\end{figure}
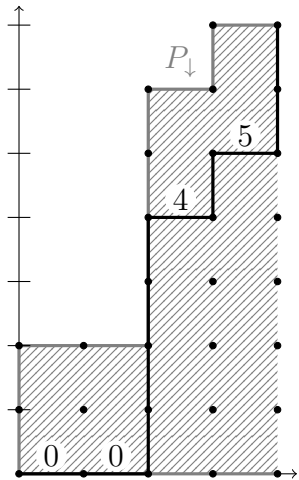

Consider the $5$-composition $(3, 1, 5, 2, 4)$ of $15$, define
similarly to $S$ the four sets $T\subseteq\sbr{1,5,2,4}$,
$U\subseteq\sbr{5,2,4,3}$, $V\subseteq\sbr{2,4,3,1}$ and
$W\subseteq\sbr{4,3,1,5}$. Then $|S|+|T|+|U|+|V|+|W|=
\left|\begin{smallmatrix} 6 & 15 \\1 & 7 \end{smallmatrix}\right|
+\left|\begin{smallmatrix} 5 & 10 \\ 1 & 8 \end{smallmatrix}\right|
+\left|\begin{smallmatrix} 8 & 28 \\ 1 & 10 \end{smallmatrix}\right|
+\left|\begin{smallmatrix} 6 & 15 \\ 1 & 6 \end{smallmatrix}\right|
+\left|\begin{smallmatrix} 5 & 10 \\ 1 & 9 \end{smallmatrix}\right|
=27+30+52+21+35=165$.  A similar construction for another
$5$-composition $c$ of $15$ gives again this number.  If, e.g., $c=(2,
1, 7, 3, 2)$, we obtain in the same manner $\left|\begin{smallmatrix}
7 & 21 \\ 1 & 9 \end{smallmatrix}\right| +\left|\begin{smallmatrix} 8
& 28 \\ 1 & 9 \end{smallmatrix}\right| +\left|\begin{smallmatrix} 9 &
36 \\ 1 & 10\end{smallmatrix}\right| +\left|\begin{smallmatrix} 4 & 6
\\ 1 & 4 \end{smallmatrix}\right| +\left|\begin{smallmatrix} 2 & 1
\\ 1 & 8 \end{smallmatrix}\right|=42+44+54+10+15=165$.
\end{example}

The result that follows generalizes this situation.

\begin{definition}
Given integers $t$, $r$ and $k$ such that $0<r<k<n$, and a $k$-composition
\hbox{$\bl=(\ell_1,\dotsc,\ell_k)\in\N^{k}$} of $n$, let
\begin{equation*}\label{the:cyl2}
s(\bl):=\sum_{i=0}^{k-1}\bbr{{\textstyle(\sum_{j=1}^r\ell_{i+j})+t,
\ell_{i+r+1},\dotsc,\ell_{i+k-1}}}\,,
\end{equation*}
where indices are to be read modulo $k$.
\end{definition}
\begin{theorem}\label{th:main}
The value of $s(\bl)$ does not depend on the $k$-composition $\bl$.
\end{theorem}
\begin{proof}
Note that, by definition, if we cyclically permute the elements of
$\bl$ the value of $s(\bl)$ does not change.  Hence, it is sufficient
to prove that, given two $k$-compositions, $\bm=(m_1,\dotsc,m_k)$ and
$\bl=(\ell_1,\dotsc,\ell_k)$ such that
$$(m_1,m_2,\dotsc,m_{k-1},m_k)=(\ell_1-1,\ell_2,\dotsc,\ell_{k-1},\ell_k+1)\,,$$
we must have $s(\bm)=s(\bl)$.

Let $s_i(\bl)=\big|\big\langle{{(\textstyle\sum_{j=1}^r\ell_{i+j})+t,
    \ell_{i+r+1},\dotsc,\ell_{i+k-1}}}\big\rangle \big|$ and define
$s_i(\bm)$ similarly. Then
$s_0(\bm)-s_0(\bl)=\sbr{L_r+t-1,\ell_{r+1},\dotsc,\ell_{k-1}}-
\sbr{L_r+t,\ell_{r+1},\dotsc,\ell_{k-1}}$, which is the negative of
$\sbr{L_{r+1}+t-1,\ell_{r+2},\dotsc,\ell_{k-1}}$ by
Lemma~\ref{lem:main}.

 In general, by subtracting and
subsequently applying Lemma~\ref{lem:main} term by term, we obtain
\begin{align*}
\sum_{i=0}^{k-1}(s_i(\bm)-s_i(\bl))=
 &\sum_{i=0}^{k-r}(s_i(\bm)-s_i(\bl))\\
=&-\bbr{L_{r+1}+t-1,\ell_{r+2},\dotsc,\ell_{k-1}}\\
&+\bbr{{\textstyle(\sum_{j=1}^{r}\ell_{j+1})+t,\ell_{r+2},\dotsc,\ell_{k-1}}}\\
&+\bbr{{\textstyle(\sum_{j=1}^{r}\ell_{j+2})+t,\ell_{r+3},\dotsc,\ell_{k-1}}}
\,\bbr{\ell_1-1}\\
&\qquad\qquad\qquad\quad\vdots\hphantom{%
\bbr{{\textstyle(\sum_{j=1}^{r}\ell_{j+k-r})+t}}}\quad\!\vdots\\
&+\bbr{{\textstyle(\sum_{j=1}^{r}\ell_{j+k-r-1})+t}}\,
\bbr{\ell_1-1,\ell_2,\dotsc,\ell_{k-r-2}}\\
&+\bbr{\ell_1-1,\ell_2,\dotsc,\ell_{k-r-1}}
\end{align*}
We prove that this number is zero by proving that the negative of the
first summand, the size of
$\mathfrak{X}=\sbr{L_{r+1}+t-1,\ell_{r+2},\dotsc,\ell_{k-1}}$, is the
sum of the other summands, each of which counts the elements with the
same image by the function $f\colon\mathfrak{X}\to[k-r]$ such that
$$f(0,x_1,\dotsc,x_{k-r-1})=
\begin{cases}
k-r,&\text{if $x_i< L_i,\ \forall\, i\leq k-r-1$;}\\[2.5pt]
\min\{i \mid x_i\geq L_i\}, &\text{otherwise.}
\end{cases}$$

First, note that $f(X)=k-r$ if and only if
$X\in\sbr{\ell_1-1,\ell_2,\dotsc,\ell_{k-r-1}}$.  If
$X\notin\sbr{\ell_1-1,\ell_2,\dotsc,\ell_{k-r-1}}$, then
\begin{align*}
f(X)\geq i&\iff\min\{i \mid x_i\geq L_i\}\geq i\\
&\iff\forall_{j<i}\,,\ x_j<L_j
\shortintertext{and hence}
f(X)= i &\iff (\forall_{j<i}\,,\ x_j<L_j) \wedge x_i\geq L_i \,.
\end{align*}
Finally,
$$(0,x_1,\dotsc,x_{k-r-1})\mapsto\Big(
(0,x_1,\dotsc,x_{i-1}),(0,x_i-L_i+1,\dotsc,x_{k-r-1}-L_i+1) \Big)$$
defines a bijection between $f^{-1}(\{i\})\subseteq\mathfrak{X}$
and the set
$$\sbr{\ell_1-1,\ell_2,\dotsc,\ell_{i-1}}
\times\big\langle{\textstyle(\sum_{j=1}^{r}\ell_{j+i})+t,
\ell_{r+i+1},\dotsc,\ell_{k-1}}\big\rangle\,.\qedhere$$
\end{proof}

\subsection{Counting parking functions and rook words with a given type}

 Recall that a parking function of length $n$ is a tuple
$\ba=(a_1,\dotsc,a_n)\in[n]^n$ such that the
$i$th entry \emph{in ascending order} is always at most $i\in[n]$.
In other words,
$$\ba\in\PF_n  \text{ if, for every $i\in[n]$, }\big|\ba^{-1}([i])\big|\geq i\,.$$

\begin{definition}\label{def:var}
Let $\ba=(a_1,\dotsc,a_n)\in\N^n$ and suppose that
$\{a_1,\dotsc,a_n\}=\{x_1,\dotsc,x_k\}$ with $x_i<x_j$ whenever $1\leq
i<j\leq k$.\\[2.5pt]
The \emph{reduced image of $\ba$} is
\begin{align*}
&\rim(\ba)=(x_1-1,\dotsc,x_k-1)\in(\N\cup\{0\})^k\,;
\shortintertext{the \emph{coimage of $\ba$} is the quotient set}
&\coim(\ba)=\big\{\eqcl{x}:=\ba^{-1}(\ba(x))\mid x\in[n]\big\}
\shortintertext{ordered by} 
&\eqcl{x}<\eqcl{y}\iff\ba(x)<\ba(y)\,.
\shortintertext{Let $\fA=(A_1,\dotsc,A_k)$ be an ordered (set)
  partition of $[n]$. The \emph{length-vector} of $\fA$ is}
&\bl(\fA)=(|A_1|,\dotsc,|A_{k-1}|)\,.
\end{align*}
In particular, the restriction of $\rim$ to $\coim(\ba)$ is always
injective and $\bl(\coim(\ba))=\rim(\ba)+(1,\dotsc,1)$.
\end{definition}

\begin{lemma}
Let $\fA$ be an ordered partition of $[n]$ with length-vector
$\bl(\fA)=(\ell_1,\dotsc,\ell_{k-1})$ and let $\ba\colon[n]\to[n]$ be
such that $\coim(\ba)=\fA$.\\ Then $\ba$ is a \emph{parking function}
if and only if
\begin{align*}
&\rim(\ba)\in\sbr{\ell_1,\dotsc,\ell_{k-1}} \shortintertext{and $\ba$
    is a \emph{run $r$ parking function} if and only if}
  &\left\{\begin{array}{l}
  \rim(\ba)=(0,1,\dotsc,r-1,x_{r+1},\dotsc,x_k)\\[2.5pt]
  (0,x_{r+1}-r,\dotsc,x_k-r)\in\sbr{(\sum_{i=1}^r\ell_i)-r,\ell_{r+1},\dotsc,
    \ell_{k-1}}
\end{array}\right.
\end{align*}

\end{lemma}\label{sec4.lemma}
\begin{proof} Follows immediately from the definitions.
\end{proof}

\noindent
Note that $\ba=(a_1,\dotsc,a_n)\in\RW_n$ if and only if, for
$i=a_1-1$, $\rim(\ba)$ belongs to the set
\begin{multline*}
\big\langle \underbrace{1,\dotsc,1}_{\hbox{\scriptsize$i$ times}},
n-k+1,\underbrace{1,\dotsc,1}_{\hbox to0pt{\scriptsize\hss$k-i-2$ times\hss}}\big\rangle= \\
= \big\{(0,1,\dotsc,i,x_{i+1},\dotsc,x_{k-1})\in[n-1]^k\mid i<x_{i+1}<\dotsb<x_{k-1}\leq n-1\big\} \,.
\end{multline*}
Hence, if we denote by $\PF_n^r$ the set of run $r$ parking functions
of length $n$, for $\fA=(A_1,\dotsc,A_k)$, according to
\eqref{eq:contas1} and by definition
\begin{align*}
&\big|\PF_n\cap\coim^{-1}(\fA)\big|= \det\limits_{1\leq i,j\leq
    k-1}\bigg( \binom{|A_1|+\dotsb+|A_i|-i+1}{j-i+1}\bigg)\\[5pt]
&\big|\PF_n^r\cap\coim^{-1}(\fA)\big|= \det\limits_{r\leq i,j\leq
    k-1}\bigg( \binom{|A_1|+\dotsb+|A_i|-i}{j-i+1}\bigg)\\[5pt]
&\big|\RW_n\cap\coim^{-1}(\fA)\big|=\binom{n-1-i}{k-1-i}
\end{align*}

\begin{definition}
Given an ordered partition $\fA$ of $[n]$ and $\ba\in[n]^n$, we say
that $\ba$ \emph{is of type $\fA$} if the coimage of $\ba$ is a
\emph{cyclic permutation of $\fA$}.
\end{definition}

\begin{theorem}\label{mthm}
Let  $\fA=(A_1,\dotsc,A_k)$ be a partition of $[n]$ and let
$1\leq r\leq n$ for a natural number $n$.
Then
\begin{itemize}
\item the number of \textbf{parking functions of type $\fA$}, as well as the number of \textbf{rook words of type $\fA$}, is $\displaystyle \binom{n}{k-1}$\,;
\item the number of \textbf{run $r$ parking functions of type $\fA$}, as well as the number of \textbf{run $r$ rook words of type $\fA$}, is $\displaystyle
  r\,\binom{n-r-1}{k-r}$\,.
\end{itemize}
\end{theorem}
\begin{proof}
Consider the two ordered $k$-compositions of $[n]$,
$\mathcal{C}=(|A_1|,|A_2|,\dotsc,|A_k|)$ and
$\mathcal{D}=(n-k+1,1,\dotsc,1)$, and apply Theorem~\ref{th:main}
with different values of $r$ and $t$.

For the first statement,
take $r=1$ and $t=0$; in the notation thereof,
$$\big|\PF_n\cap\coim^{-1}(\fA)\big|=s(\mathcal{C})\ =\
\sum_{i=0}^{k-1}\binom{n-1-i}{k-1-i}
\ =\ s(\mathcal{D})=\big|\RW_n\cap\coim^{-1}(\fA)\big|\,.$$

For the second statement, by taking $r=1,\dotsc,k$ and $t=-r$
we obtain that
$$\big|\PF_n^r\cap\coim^{-1}(\fA)\big|=
s(\mathcal{D})=r\big|\big\langle n-k, \underbrace{1,\dotsc,1}_{\hbox
  to0pt{\scriptsize\hss$k-r-1$ times\hss}}\big\rangle\big| +0\,,$$
since, for $\ell_1=n-k+1$ and $\ell_2=\dotsb=\ell_k=1$,
$$
\nbr{{\textstyle(\sum_{j=1}^r\ell_{i+j})-r,
    \ell_{i+r+1},\dotsc,\ell_{i+k-1}}}=
\begin{cases}
\nbr{n-k,1,\dotsc,1},& \text{if $i\in[r]$;}\\
0,&\text{otherwise.}
\end{cases}
$$
This shows that the number of run $r$ parking functions
    of type $\fA$ is $r\,\binom{n-r-1}{k-r}$, since
$$\sbr{n-k,1,\dotsc,1}=\big\{(x_1,\dotsc,x_{k-r})\mid 0<x_1<\dotsb<x_{k-r}
\leq n-r-1\big\}\,.$$

Finally, note that, for example, all the type~$\fA$ elements
$\ba=(a_1,a_2,\dotsc,a_n)\in[n]^n$ with $a_1=1$ share the same
coimage, and that there are $\binom{n-r-1}{k-r}$ such rook words with
run $r$, since they are determined by the last $k-r$ strictly
increasing coordinates of $\rim(\ba)$, all of them greater than $r$
and less than $n$. The same happens if $a_1=i$ for $1\leq i\leq r$,
and $a_i$ cannot be greater than $r$, by definition.
\end{proof}

We note that the first part of Theorem~\ref{mthm} can be obtained
directly from \cite[Cyclic Lemma]{LRW}, where the following bijection
is defined. Let $\bb=\ba$ if $\ba\in\PF_n\cap\RW_n$ and, if
$\ba\in\PF_n\setminus\RW_n$ and
$m=\max\big([a_1]\setminus\ba([n])\big)$, let
$\bb=(b_1,\dotsc,b_n)\in[n]^n$ be such that
$$a_i\equiv b_i+m\pmod{n}\,;$$
then $\ba\mapsto\bb$ defines a
bijection between the set of parking functions and the set of rook
words \emph{of a given type}.

\begin{theorem}\label{pf2rw}
For every $n\in\N$,
$$\RP_n(t)=\RR_n(t)\,.$$
\end{theorem}
\begin{proof} Follows immediately from Theorem~\ref{mthm}.\end{proof}

\section{Counting rook words with a given run}\label{final}

Given positive integers $n$ and $r$ such that $r\leq n$, let
\begin{align*}
&\RW_n^r = \{ f \in \RW_n \mid \run(f) = r \}
\shortintertext{and for $\ba=(a_1,\dotsc,a_n)\in\RW_n^r$ let}
&\br=\br(\ba)=(i_1,\dotsc,i_r)
\end{align*}
where $i_j=\min\{i\in[n]\mid a_i=j\}$ for every $j\in[r]$
(compare with the definition of $\Run$ in page~\pageref{defrun}).

\begin{theorem}\label{Cfinal}
For every integers $1\leq r\leq n$,
\begin{align}
[t^r]\big(\LT_n(t)\big)&=
r! \sum_{e_1 + \dotsb + e_r = n-r} (n-1)^{e_1} (n-2)^{e_2}
\dotsb (n-r)^{e_r} \label{firid}\\
&= r\ \sum_{j=0}^{r-1} (-1)^j \binom{r-1}{j} (n-1-j)^{n-1}\label{secid}\,.
\end{align}
\end{theorem}

\begin{proof}
We have seen before that
$$[t^r]\big(\LT_n(t)\big)=[t^r]\big(\RR_n(t)\big)=\big|\RW_n^r\big| .$$
Given $\ba\in\RW_n^r$ and $\pi \in \mathfrak{S}_r$, let $\pi \ba$ be
the element of $[n]^n$ defined by
$$ (\pi \ba)(j)=
\begin{cases}
\pi (a_j) & \text{if $a_j \leq r$\,;} \\ a_j & \text{if $j > r$\,.}
\end{cases}
$$
Note that $\pi \ba\in\RW_n^r$ if and only if $\ba\in\RW_n^r$. Owing
to this, the left-hand side of \eqref{firid}
is equal to $r!$ times the number of elements of
$$ A = \{ \ba\in\RW_n^r \mid \br(\ba)=(i_1,\dotsc,i_r) \ \text{with}
\ 1=i_1 < i_2 < \dotsb < i_r \}. $$
Now,  $\ba=(a_1,\dotsc,a_n) \in A$ if and only if, for every $1\leq\ell\leq n$,
\begin{itemize}
\item $a_\ell \notin\{i_{j+1},\dotsc,i_r,r+1\}$, if $i_j<\ell < i_{j+1}$ for some
$j\in[r-1]$\,;
\item $a_\ell \neq r+1$, if $\ell > i_r$.
\end{itemize}
This gives \eqref{firid} for $e_j=i_{r+2-j}-i_{r+1-j}-1$ with $1<j\leq r$,  
and $e_1=n-i_r$.

We note that the right-hand side of \eqref{secid} is, by the
Inclusion-Exclusion~Principle, $r$ times the number of elements of
$$B=\big\{f\colon[n-1]\to[n-1]\mid [r-1]\subseteq f([n-1])\big\}\,.$$

Given $\ell\in[n]$ with $\ell\leq r < n$, consider the bijection
$\varphi_{\ell}\colon[n]\setminus\{r+1\}\to[n-1]$ such that
$$\varphi_{\ell}(j)=\begin{cases}
j,&\text{if $j<\ell$\,;}\\
r,&\text{if $j=\ell$\,;}\\
j-1,&\text{if $j>\ell$\,.}
\end{cases}$$
and note that $[r-1]\subseteq\varphi_{\ell}([r])$.  Now,
$F(a_1,\dotsc,a_n)=(a_1,\varphi_{a_1}(a_2),\dotsc,\varphi_{a_1}(a_r))$
clearly defines a bijection from $\RW_n^r$ to $[r]\times B$.
\end{proof}

\printbibliography

\end{document}